\documentclass[11pt]{amsart}

\usepackage{amsmath,amssymb,amsthm}
\usepackage[colorlinks=true,linkcolor=blue,citecolor=blue,urlcolor=blue]{hyperref}

\numberwithin{equation}{section}

\newtheorem{theorem}{Theorem}[section]
\newtheorem{lemma}[theorem]{Lemma}
\newtheorem{proposition}[theorem]{Proposition}
\theoremstyle{remark}
\newtheorem{remark}[theorem]{Remark}

\DeclareMathOperator{\sgn}{sgn}
\newcommand{\Ba}{\mathbf{a}}
\newcommand{\Bx}{\mathbf{x}}
\newcommand{\Bv}{\mathbf{v}}
\newcommand{\Bw}{\mathbf{w}}
\newcommand{\Beps}{\boldsymbol{\epsilon}}

\title[A convexity proof of Pohst's inequality]{A convexity proof of Pohst's~inequality}

\author[Scott Duke Kominers]{Scott Duke Kominers}
\address{Harvard Business School; Department of Economics and Center of Mathematical Sciences and Applications, Harvard University; and a16z crypto}
\email{kominers@fas.harvard.edu}

\thanks{I used LLMs to assist with analysis and computations in the preparation of this article, particularly GPT-5.5 Pro and Claude~4.8 Opus (both accessed in part via Poe with the support of Quora, where I am an advisor). In particular, the inductive convexity reduction used here was initially suggested by Claude~4.8 Opus, and GPT-5.5 Pro identified a crucial refinement. I especially appreciate a thorough review from Refine.ink.
The problem, final methods, and eventual written form are my own; and of course any errors remain my responsibility. This work was conducted while I was visiting the Technological Innovation, Entrepreneurship, and Strategic Management (TIES) Group at the MIT Sloan School of Management; I greatly appreciate their hospitality.}


\begin{document}

\begin{abstract}
We give a short analytic proof of Raposo's signed refinement of Pohst's inequality.
\end{abstract}

\maketitle

\section{Introduction}

Let $N\ge 1$ and let $y_1,\dots,y_N\in\mathbb{R}\setminus\{0\}$ satisfy
\[
|y_1|\le |y_2|\le \cdots \le |y_N|.
\]
We set
\[
P_N(y_1,\dots,y_N)
=
\prod_{1\le i<j\le N}
\left|1-\frac{y_i}{y_j}\right|,
\]
and let $p$ and $m$ denote the number of positive and negative entries, respectively, so that $p+m=N$. We give a new proof of the following bound.

\begin{theorem}[Raposo~\cite{RaposoRefinement2025}]\label{thm:main}
For every $N\ge 1$ and every admissible tuple as above,
\[
P_N(y_1,\dots,y_N)\le 2^{\min(p,m)}.
\]
In particular $P_N\le 2^{\lfloor N/2\rfloor}$, since $\min(p,m)\le\lfloor N/2\rfloor$.
\end{theorem}

Products of the shape $P_N$ arise when one bounds the regulator $R_K$ of a number field $K$ from below in terms of its discriminant $D_K$; for totally real fields this was the motivation of Remak and Pohst, and it remains central to the determination of fields of small regulator~\cite{Friedman1989,ADF2016,FR2019}. Remak~\cite{Remak1952} proved, for complex variables of modulus at most $1$, the bound $P_N\le N^{N/2}$. Pohst~\cite{Pohst1977} conjectured the sharp constant $2^{\lfloor N/2\rfloor}$ in the real case and verified it for $N\le 11$ by a computer-assisted argument, factoring the product according to each sign pattern into blocks to which four elementary inequalities apply. Bertin~\cite{Bertin1996} attempted a proof for all $N$, but unfortunately that argument was incomplete. The inequality was finally established for all $N$, independently, by Raposo~\cite{RaposoProof2024} and by Battistoni and Molteni~\cite{BM2021}; the signed refinement with exponent $\min(p,m)$ stated in Theorem~\ref{thm:main} is due to Raposo~\cite{RaposoRefinement2025}, and related generalized bounds in the one-complex-place setting were obtained by Battistoni and Molteni~\cite{BM2025}.

The known proofs of the sharp real inequality, and of Raposo's signed refinement, are combinatorial in nature: one partitions the triangular array of index pairs $(i,j)$ into blocks of one to four factors and applies Pohst's four elementary inequalities block-by-block, the difficulty lying in the inductive construction of a suitable partition (a ``good partition'' in~\cite{RaposoRefinement2025}, or ``graphical schemes'' transformed by elementary moves in~\cite{BM2021,BM2025}). Raposo notes that any improvement beyond the four elementary inequalities ``would need more tools''~\cite[p.~749]{RaposoRefinement2025}.

The proof we present is of a different character. It is analytic and turns on one convexity observation. After a scale-invariant change of variables (Section~\ref{sec:reform}) the product $P_N$ becomes a continuous function $P_N(\Bx;\Beps)$ on the cube $[0,1]^{N-1}$, where $\Beps$ records the signs. The key point (Section~\ref{sec:nointerior}) is that \emph{this function has no interior local maximum} whenever both signs occur. Passing to logarithmic coordinates, consider the perturbation that translates the positive entries one way and the negative entries the other, uniformly. Each same-sign pairwise distance is then unchanged, while each opposite-sign distance varies linearly; since the opposite-sign logarithmic factors are strictly convex in that distance, the restricted function is strictly convex and thus admits no interior maximum. Hence the maximum lies on the boundary of the cube, where two elementary reductions, a \emph{collapse} (Section~\ref{sec:collapse}) and a \emph{split} (Section~\ref{sec:split}), lower the dimension and an induction closes the bound (Section~\ref{sec:completion}). 

\section{Reformulation on a compact cube}\label{sec:reform}

The product $P_N$ is invariant under $y_i\mapsto c\,y_i$ for any nonzero $c$; a positive $c$ fixes the signs and a negative $c$ reverses them all, leaving $\min(p,m)$ unchanged. For each $i$ put
\[
\epsilon_i=\sgn(y_i)\in\{\pm1\},
\qquad
x_k=\frac{|y_k|}{|y_{k+1}|}\in(0,1]
\quad(1\le k\le N-1),
\]
and for $1\le i<j\le N$ set
\[
X_{ij}=X_{ij}(\Bx)=\prod_{k=i}^{j-1}x_k.
\]
Telescoping gives $X_{ij}=|y_i|/|y_j|$, whence
\[
\frac{y_i}{y_j}=\epsilon_i\epsilon_j X_{ij},
\qquad
\left|1-\frac{y_i}{y_j}\right|
=
\bigl|1-\epsilon_i\epsilon_j X_{ij}\bigr|.
\]
Since $X_{ij}\in[0,1]$, the quantity $1-\epsilon_i\epsilon_j X_{ij}$ is already nonnegative: it equals $1-X_{ij}\in[0,1]$ for a same-sign pair and $1+X_{ij}\in[1,2]$ for an opposite-sign pair, so the absolute value may be dropped. For a fixed sign vector $\Beps\in\{\pm1\}^N$ define, for all $\Bx\in[0,1]^{N-1}$,
\begin{equation}\label{eq:Pdef}
P_N(\Bx;\Beps)
=
\prod_{1\le i<j\le N}
\bigl(1-\epsilon_i\epsilon_j X_{ij}(\Bx)\bigr);
\end{equation}
this is a polynomial in $\Bx$, hence it is continuous on the closed cube $[0,1]^{N-1}$.

Admissible tuples with the prescribed signs correspond exactly to points $\Bx\in(0,1]^{N-1}$: given such an $\Bx$ one recovers a tuple by setting $|y_N|=1$ and $|y_i|=x_i x_{i+1}\cdots x_{N-1}$, which are nonzero precisely because we always have $x_k>0$. Points with some $x_k=0$ are boundary points of this compactification and need not correspond to admissible tuples; we introduce them only to locate the maximum of the continuous extension. Thus Theorem~\ref{thm:main} follows once we prove
\[
P_N(\Bx;\Beps)\le 2^{\min(p,m)}
\qquad\text{for all }\Bx\in[0,1]^{N-1}.
\]

\begin{proposition}\label{prop:SN}
For every $N\ge 1$, every sign vector $\Beps\in\{\pm1\}^N$ with sign counts $(p,m)$, and every $\Bx\in[0,1]^{N-1}$,
\[
P_N(\Bx;\Beps)\le 2^{\min(p,m)}.
\]
\end{proposition}

We prove Proposition~\ref{prop:SN} by strong induction on $N$.

\section{Base cases and trivial sign patterns}

First, we record the base cases; we need two of them in this instance because part of the argument relies on a bifurcation of the indices.

\begin{lemma}\label{lem:base}
Proposition~\ref{prop:SN} holds for $N=1$ and $N=2$. Moreover, for any $N$, if all signs are equal then $P_N(\Bx;\Beps)\le 1=2^0$.
\end{lemma}

\begin{proof}
For $N=1$ the product is empty, so $P_1=1=2^0$. For $N=2$, $P_2(\Bx;\Beps)=1-\epsilon_1\epsilon_2 x_1$; if the signs are equal then $P_2=1-x_1\le 1=2^0$ with $\min(p,m)=0$, and if they are opposite then $P_2=1+x_1\le 2=2^1$ with $\min(p,m)=1$. Finally, if all signs are equal then every factor of \eqref{eq:Pdef} has the form $1-X_{ij}\in[0,1]$, so $P_N\le 1=2^0$, and $\min(p,m)=0$.
\end{proof}

\section{The boundary reductions}

There are two boundary faces to treat: $x_k=1$, where adjacent absolute values coincide, and $x_k=0$, where the product splits.

\subsection{Collapse at \texorpdfstring{$x_k=1$}{x[k]=1}}\label{sec:collapse}

\begin{lemma}[Collapse]\label{lem:collapse}
Let $N\ge 3$, assume Proposition~\ref{prop:SN} for $N-2$, and let $\Bx\in[0,1]^{N-1}$ with $x_k=1$ for some $k$. Then $P_N(\Bx;\Beps)\le 2^{\min(p,m)}$.
\end{lemma}

\begin{proof}
If $\epsilon_k=\epsilon_{k+1}$, the factor for the pair $(k,k+1)$ is
\[
1-\epsilon_k\epsilon_{k+1}X_{k,k+1}=1-x_k=0,
\]
so $P_N=0$ and the bound is immediate. We may therefore assume that $\epsilon_{k+1}=-\epsilon_k$; then the $(k,k+1)$ factor is $2$.

We let
\[
\iota:\{1,\dots,N-2\}\longrightarrow \{1,\dots,N\}\setminus\{k,k+1\}
\]
be the increasing deletion map,
\[
\iota(r)=
\begin{cases}
r, & r<k,\\
r+2, & r\ge k.
\end{cases}
\]
Put $\widehat\epsilon_r=\epsilon_{\iota(r)}$. Define $\widehat{\Bx}\in[0,1]^{N-3}$ by
\[
\widehat x_\ell=
\begin{cases}
x_\ell, & \ell\le k-2,\\
x_{k-1}x_{k+1}, & \ell=k-1\ \text{and}\ 1<k<N-1,\\
x_{\ell+2}, & \ell\ge k,
\end{cases}
\]
where the middle case is omitted when $k=1$ or $k=N-1$. Thus $\widehat{\Bx}$ is precisely the adjacent-ratio vector obtained after deleting the two equal-modulus entries in positions $k$ and $k+1$.

For $1\le r<s\le N-2$, write
\[
\widehat X_{rs}=\prod_{\ell=r}^{s-1}\widehat x_\ell.
\]
We claim that
\[
\widehat X_{rs}=X_{\iota(r),\iota(s)}.
\]
Indeed, if $s<k$, then both indices lie before the deleted block and
\[
\widehat X_{rs}=x_r\cdots x_{s-1}=X_{r,s}=X_{\iota(r),\iota(s)}.
\]
If $r\ge k$, then both indices lie after the deleted block and
\[
\widehat X_{rs}=x_{r+2}\cdots x_{s+1}=X_{r+2,s+2}=X_{\iota(r),\iota(s)}.
\]
Finally, if $r<k\le s$, then the pair straddles the deleted block, and
\[
\widehat X_{rs}
=
x_r\cdots x_{k-2}\,(x_{k-1}x_{k+1})\,x_{k+2}\cdots x_{s+1}
=
x_r\cdots x_{s+1}
=
X_{r,s+2},
\]
because $x_k=1$. This is again $X_{\iota(r),\iota(s)}$.

It follows that the factors of \eqref{eq:Pdef} involving neither $k$ nor $k+1$ are exactly the factors of
\[
P_{N-2}(\widehat{\Bx};\widehat{\Beps}).
\]

Finally, we must account for the factors involving one of the deleted indices. For $a<k$, as $x_k=1$, we have $X_{a,k+1}=X_{a,k}$, and hence
\[
\bigl(1-\epsilon_a\epsilon_k X_{a,k}\bigr)
\bigl(1-\epsilon_a\epsilon_{k+1}X_{a,k+1}\bigr)
=
\bigl(1-\epsilon_a\epsilon_k X_{a,k}\bigr)
\bigl(1+\epsilon_a\epsilon_k X_{a,k}\bigr)
=
1-X_{a,k}^2.
\]
Similarly, for $b>k+1$ we have $X_{k,b}=X_{k+1,b}$, and
\[
\bigl(1-\epsilon_k\epsilon_b X_{k,b}\bigr)
\bigl(1-\epsilon_{k+1}\epsilon_b X_{k+1,b}\bigr)
=
1-X_{k+1,b}^2.
\]
Therefore, we have
\[
P_N(\Bx;\Beps)
=
2
\prod_{a<k}\bigl(1-X_{a,k}^2\bigr)
\prod_{b>k+1}\bigl(1-X_{k+1,b}^2\bigr)
P_{N-2}(\widehat{\Bx};\widehat{\Beps})
\le
2\,P_{N-2}(\widehat{\Bx};\widehat{\Beps}),
\]
as all the squared factors lie in $[0,1]$.

Deleting $\epsilon_k,\epsilon_{k+1}$ removes one plus sign and one minus sign, so the new sign counts are $(p-1,m-1)$.
Moreover, 
\[
\min(p-1,m-1)=\min(p,m)-1.
\]
By the inductive hypothesis, we have
\[
P_{N-2}(\widehat{\Bx};\widehat{\Beps})
\le
2^{\min(p,m)-1};
\]
hence, $P_N(\Bx;\Beps)\le 2^{\min(p,m)}$.
\end{proof}

\subsection{Split at \texorpdfstring{$x_k=0$}{x[k]=0}}\label{sec:split}

\begin{lemma}[Split]\label{lem:split}
Let $N\ge 2$ and $1\le k\le N-1$, and assume Proposition~\ref{prop:SN} for $k$ and for $N-k$. If $\Bx\in[0,1]^{N-1}$ with $x_k=0$, then $P_N(\Bx;\Beps)\le 2^{\min(p,m)}$.
\end{lemma}

\begin{proof}
For a crossing pair $i\le k<j$ the monomial $X_{ij}$ contains the factor $x_k=0$, so that factor of \eqref{eq:Pdef} equals $1$; every crossing factor therefore disappears, and the product splits as
\[
P_N(\Bx;\Beps)=P^{(L)}\,P^{(R)},
\]
where
\[
\begin{aligned}
P^{(L)}&=P_k\bigl((x_1,\dots,x_{k-1});(\epsilon_1,\dots,\epsilon_k)\bigr),\\
P^{(R)}&=P_{N-k}\bigl((x_{k+1},\dots,x_{N-1});(\epsilon_{k+1},\dots,\epsilon_N)\bigr),
\end{aligned}
\]
are the products over the index blocks $\{1,\dots,k\}$ and $\{k+1,\dots,N\}$ (for $k=1$ or $k=N-1$ the corresponding coordinate vector is empty). With block sign counts $(p_L,m_L)$ and $(p_R,m_R)$ one has $p=p_L+p_R$ and $m=m_L+m_R$, and the inductive hypothesis gives $P^{(L)}\le 2^{\min(p_L,m_L)}$, $P^{(R)}\le 2^{\min(p_R,m_R)}$. Since
\[
\begin{aligned}
\min(p_L,m_L)+\min(p_R,m_R)&\le p_L+p_R=p,\\
\min(p_L,m_L)+\min(p_R,m_R)&\le m_L+m_R=m,
\end{aligned}
\]
we see that $\min(p_L,m_L)+\min(p_R,m_R)\le\min(p,m)$, and therefore we have $P_N(\Bx;\Beps)\le 2^{\min(p,m)}$.
\end{proof}

\subsection{Boundary reduction synthesis}\label{sec:boundary}

Lemma~\ref{lem:collapse} and Lemma~\ref{lem:split} combine to provide a unified inductive conclusion on the boundary.

\begin{proposition}\label{prop:boundary}
Let $N\ge 3$, and assume Proposition~\ref{prop:SN} has been proven in every positive dimension strictly less than $N$. Then for every sign vector $\Beps\in\{\pm1\}^N$ with sign counts $(p,m)$ and every boundary point $\Bx\in\partial[0,1]^{N-1}$, we have
\[
P_N(\Bx;\Beps)\le 2^{\min(p,m)}.
\]
\end{proposition}

\begin{proof}
Since $\Bx$ lies on the boundary of the cube, some coordinate $x_k$ is equal to $0$ or to $1$. If $x_k=0$, then Lemma~\ref{lem:split} applies because the required dimensions $k$ and $N-k$ are both strictly smaller than $N$. If $x_k=1$, then Lemma~\ref{lem:collapse} applies because the required dimension $N-2$ is strictly smaller than $N$. Thus in either case we obtain the desired bound.
\end{proof}

\section{The analytic lemma: no interior local maximum}\label{sec:nointerior}

Our boundary reductions, summarized in Proposition~\ref{prop:boundary}, are conditional on the lower-dimensional inductive hypotheses. We now provide the independent analytic ingredient that prevents a maximizer from lying in the interior.

\begin{lemma}[No interior local maximum]\label{lem:nointerior}
Let $N\ge 2$ and suppose $\Beps$ contains at least one $+1$ and at least one $-1$. Then $P_N(\cdot\,;\Beps)$ has no local maximum in the open cube $(0,1)^{N-1}$.
\end{lemma}

\begin{proof}
Fix $\Bx\in(0,1)^{N-1}$. For $1\le k\le N-1$ put $v_k=-\log x_k>0$, so $x_k=e^{-v_k}$. For $\Bw\in(0,\infty)^{N-1}$ and $i<j$, set
\[
d_{ij}(\Bw)=w_i+\cdots+w_{j-1},
\]
and define
\[
G(\Bw)
=
\log P_N\bigl((e^{-w_1},\dots,e^{-w_{N-1}});\Beps\bigr)
=
\sum_{1\le i<j\le N}
\log\bigl(1-\epsilon_i\epsilon_j e^{-d_{ij}(\Bw)}\bigr).
\]
This is smooth on $(0,\infty)^{N-1}$, since all its factors are strictly positive there.

It remains to choose the one-dimensional variation. Choose logarithmic positions $z_i$ with $v_k=z_{k+1}-z_k$, so that $d_{ij}=z_j-z_i$. A variation $z_i\mapsto z_i+s b_i$ changes $d_{ij}$ by $s(b_j-b_i)$. To keep every same-sign distance fixed, the numbers $b_i$ must be constant on each sign class; to make every opposite-sign distance vary nontrivially, those two constants must be distinct. Up to adding a common constant and rescaling, the canonical choice is $b_i=\epsilon_i$.

Thus we perturb by $z_i\mapsto z_i+s\epsilon_i$. The induced variation of the adjacent logarithmic ratios is
\[
v_k=z_{k+1}-z_k
\longmapsto
v_k+s(\epsilon_{k+1}-\epsilon_k).
\]
Accordingly, define $\Ba=(a_1,\dots,a_{N-1})$ by
\[
a_k=\epsilon_{k+1}-\epsilon_k.
\]
In particular $\Ba$ is generally \emph{not} a coordinate direction, which is what lets it detect saddle behavior that a coordinatewise analysis misses (cf.~Remark~\ref{rem:example}). Since we have $v_k>0$ for all $k$, there is $\delta>0$ with $\Bv+s\Ba\in(0,\infty)^{N-1}$ for $|s|<\delta$. Along $\Bv(s)=\Bv+s\Ba$,
\[
d_{ij}(\Bv(s))
=
\sum_{\ell=i}^{j-1}\bigl(v_\ell+sa_\ell\bigr)
=
d_{ij}(\Bv)+s\sum_{\ell=i}^{j-1}(\epsilon_{\ell+1}-\epsilon_\ell)
=
d_{ij}(\Bv)+s(\epsilon_j-\epsilon_i),
\]
the inner sum telescoping, so the shift depends only on the endpoint signs. Same-sign pairs ($\epsilon_i=\epsilon_j$) keep $d_{ij}(\Bv(s))$ constant; opposite-sign pairs have $\epsilon_j-\epsilon_i=\pm 2$, and their factor is $h(d_{ij}(\Bv(s)))$ with
\[
h(u)=\log(1+e^{-u}),
\qquad
h''(u)=\frac{e^{-u}}{(1+e^{-u})^2}=\frac{1}{4\cosh^2(u/2)}>0 .
\]
Therefore $\Phi(s):=G(\Bv+s\Ba)$ has
\[
\Phi''(s)
=
4\sum_{\substack{1\le i<j\le N\\ \epsilon_i\ne\epsilon_j}}
h''\bigl(d_{ij}(\Bv(s))\bigr)
=
\sum_{\substack{1\le i<j\le N\\ \epsilon_i\ne\epsilon_j}}
\frac{1}{\cosh^2\!\bigl(d_{ij}(\Bv(s))/2\bigr)}
>0,
\]
the sum being nonempty because both signs occur. Thus $\Phi$ is strictly convex on $(-\delta,\delta)$, and a strictly convex function has no interior local maximum: for small $t>0$,
\[
\Phi(0)<\tfrac12\bigl(\Phi(t)+\Phi(-t)\bigr)\le\max\{\Phi(t),\Phi(-t)\}.
\]
Hence $G$ has no local maximum at $\Bv$. Since $\Bx\mapsto\Bv$ is a homeomorphism of $(0,1)^{N-1}$ onto $(0,\infty)^{N-1}$, the logarithm is strictly increasing, and $P_N>0$ on the open cube, $P_N(\cdot\,;\Beps)$ has no local maximum in $(0,1)^{N-1}$.
\end{proof}

\section{Completion of the induction}\label{sec:completion}

We may now complete our argument.

\begin{proof}[Proof of Proposition~\ref{prop:SN} and Theorem~\ref{thm:main}]
We argue by strong induction on $N$. The cases $N=1,2$ are Lemma~\ref{lem:base}. Fix $N\ge 3$ and assume that Proposition~\ref{prop:SN} holds for every positive integer up to $N-1$.

Let $\Beps\in\{\pm1\}^N$ have counts $(p,m)$. The polynomial $P_N(\cdot\,;\Beps)$ is continuous on the compact cube $[0,1]^{N-1}$ and so attains its maximum at some $\Bx^\ast$. If $\min(p,m)=0$, then all signs are equal and Lemma~\ref{lem:base} gives $P_N(\Bx^\ast;\Beps)\le 1$. Otherwise both signs occur. A maximizer in the open cube would be a local maximum there, contradicting Lemma~\ref{lem:nointerior}; hence $\Bx^\ast$ lies on the boundary of the cube. The strong inductive hypothesis is precisely what we need in order to apply Proposition~\ref{prop:boundary}, so we have
\[
P_N(\Bx^\ast;\Beps)\le 2^{\min(p,m)}.
\]
It follows that
\[
P_N(\Bx;\Beps)
\le
P_N(\Bx^\ast;\Beps)
\le
2^{\min(p,m)}
\]
for all $\Bx$, proving Proposition~\ref{prop:SN}. Theorem~\ref{thm:main} then follows via the reformulation given in Section~\ref{sec:reform}.
\end{proof}

\section{Discussion}\label{sec:remarks}

We close with several remarks on the argument and potential extensions.

\begin{remark}[Per-coordinate maxima are not enough]\label{rem:example}
For $\Beps=(+,-,+)$,
\[
P_3(x_1,x_2)=(1+x_1)(1+x_2)(1-x_1x_2),
\]
and the stationary equations $\tfrac{1}{1+x_1}=\tfrac{x_2}{1-x_1x_2}$, $\tfrac{1}{1+x_2}=\tfrac{x_1}{1-x_1x_2}$ have the interior solution $(\tfrac12,\tfrac12)$, where $P_3=\tfrac{27}{16}$; this point is a maximum along each coordinate axis, but it reflects a saddle point. Along the sign-separating direction of Lemma~\ref{lem:nointerior}, namely $x_1(s)=\tfrac12 e^{2s}$, $x_2(s)=\tfrac12 e^{-2s}$ (which freezes $x_1x_2=\tfrac14$),
\[
P_3(s)=\Bigl(1+\tfrac12 e^{2s}\Bigr)\Bigl(1+\tfrac12 e^{-2s}\Bigr)\tfrac34=\tfrac34\Bigl(\tfrac54+\cosh (2s)\Bigr),
\]
which has a strict minimum at $s=0$. The maximum of $P_3$ is attained on the boundary, e.g., $P_3(1,0)=2=2^{\min(2,1)}$; this is why Lemma~\ref{lem:nointerior} must use a non-coordinate direction.
\end{remark}

\begin{remark}[Heuristic interpretation]\label{rem:heuristic}
In the coordinates $z_i=\log|y_i|$, viewing the entries as charges of sign $\epsilon_i$ on a line and $\log P_N$ as an energy, the perturbation $z_i\mapsto z_i+s\epsilon_i$ fixes all same-charge distances and changes each opposite-charge distance linearly, with slope $\pm2$. Along the signed translation just described, the energy is strictly convex. Thus from any interior configuration, at least one of the two sufficiently small opposite signed translations strictly increases the product. In this sense, our argument can be interpreted as a rigorous form of a charged-particle heuristic presented by Raposo~\cite{RaposoRefinement2025}.
\end{remark}

\begin{remark}[Relation to Bertin's attempted proof]\label{rem:bertin}
Bertin~\cite{Bertin1996} also pursued a boundary-reduction strategy, using a determinant representation and a simplex argument applied row-by-row. The point needing care is that the entries in such a representation are not independent row variables; in the ordered real specialization, they are tied together through the same adjacent-ratio parameters. Consequently, a row-wise simplex extremality argument does not by itself justify a global reduction of the underlying variables to boundary values---equivalently, of the signed adjacent ratios to $0,\pm1$, or of the $x_k$ to $0$ or $1$ with the signs fixed. Our Lemma~\ref{lem:nointerior} provides a valid replacement for that boundary-reduction step, directly in the compactified adjacent-ratio variables. Once the maximum is known to lie on the boundary of the cube, the split and collapse lemmas complete the induction.
\end{remark}

\begin{remark}[Sharpness]\label{rem:sharp}
As a bound depending only on the counts $(p,m)$, the exponent $\min(p,m)$ is optimal. Put $r=\min(p,m)$ and take the sign pattern consisting of $r$ adjacent opposite-sign pairs followed by $|p-m|$ equal signs; setting the ratio within each pair to $1$ and letting the separating ratios tend to $0$ drives $P_N$ to $2^{r}$.

For a \emph{fixed} sign pattern, however, the bound need not be sharp. For example, for $\Beps=(+,+,-,-)$ the compactified maximum is $2$, and hence the supremum over admissible tuples is $2$, although $\min(p,m)=2$. Indeed, by Lemma~\ref{lem:nointerior} any compactified maximum lies on the boundary. On the faces $x_1=1$ and $x_3=1$ the product vanishes, while on the split faces $x_i=0$ Proposition~\ref{prop:SN} in lower dimensions gives at most $2$. On the remaining collapse face $x_2=1$,
\[
P_4(x_1,1,x_3;\Beps)
=
2(1-x_1^2)(1-x_3^2)(1+x_1x_3)
\le 2,
\]
because we have
\[
(1-x_1^2)(1-x_3^2)\le (1-x_1x_3)^2
\]
and, with $q=x_1x_3\in[0,1]$,
\[
(1-q)^2(1+q)\le 1.
\]
The value $2$ is obtained at the compactified vertex $(0,1,0)$ and is approached by admissible points. This is an instance of Raposo's proposed refinement, which we discuss next.
\end{remark}

\begin{remark}[Vertex obstruction for Raposo's conjectured $2^M$~improvement]\label{rem:2M}
Raposo~\cite{RaposoRefinement2025} conjectures the sharper bound
\[
P_N\le 2^{M(\Beps)},
\]
where $M(\Beps)\le\min(p,m)$ is the maximal number of disjoint adjacent opposite-sign pairs in~$\Beps$, equivalently the matching number of the graph on $\{1,\dots,N\}$ whose edges are the adjacent positions of opposite sign.

The vertices of the compactified cube give the natural target value. Indeed, evaluate $P_N(\Bx;\Beps)$ at a vertex $\Bx\in\{0,1\}^{N-1}$. A run of two or more consecutive $1$s locks three or more entries to a common modulus, and among three signs two are equal; hence some same-sign factor vanishes. Thus a nonzero vertex has only isolated $1$s, and each isolated $1$ must lock an adjacent opposite-sign pair.

The isolated $1$s therefore form a matching of adjacent opposite-sign edges. All other factors are split apart by intervening $0$s, and hence
\[
P_N(\Bx;\Beps)
=
2^{\#\{\text{locked pairs}\}}.
\]
It follows that
\[
\max_{\Bx\in\{0,1\}^{N-1}}
\bigl\{P_N(\Bx;\Beps)\bigr\}
=
2^{M(\Beps)};
\]
thus, Raposo's conjecture corresponds to the assertion that the maximum over the whole cube is no larger than this vertex maximum. 

Our convexity approach gives the first part of such a strategy: Lemma~\ref{lem:nointerior} rules out an interior maximizer. On a split face $x_k=0$, the conjectural bound is also compatible with induction, since the matching numbers $M_L$ and $M_R$ of the two induced sign strings satisfy
\[
M_L+M_R\le M(\Beps).
\]

The remaining difficulty concerns an opposite-sign collapse face $x_k=1$. There the collapse calculation gives
\[
P_N(\Bx;\Beps)
=
2
\prod_{a<k}\bigl(1-X_{a,k}^2\bigr)
\prod_{b>k+1}\bigl(1-X_{k+1,b}^2\bigr)
P_{N-2}(\widehat{\Bx};\widehat{\Beps}).
\]
For the coarser bound $2^{\min(p,m)}$, the extra neutral factors $1-X^2$ are harmless because they are at most $1$. For the conjectural $2^{M(\Beps)}$ bound, however, induction on the collapsed sign vector is not automatically sufficient because deleting the collapsed opposite-sign pair can create a new adjacent opposite-sign pair across the gap, so $M(\widehat{\Beps})$ need not be at most $M(\Beps)-1$. For instance, collapsing the middle pair in the sign pattern $(+,+,-,-)$ leaves $(+,-)$, and both matching numbers are equal to $1$.

Thus the neutral factors must provide additional compensation on such collapse faces, but they no longer have the signed product structure to which Lemma~\ref{lem:nointerior} applies. This collapse-face phenomenon appears to be the crux of Raposo's conjectural improvement, and it may also be relevant to Battistoni and Molteni's conjectured extensions~\cite{Battistoni2021,BM2025}.
\end{remark}

\providecommand{\bysame}{\leavevmode\hbox to3em{\hrulefill}\thinspace}
\providecommand{\MR}{\relax\ifhmode\unskip\space\fi MR }
% \MRhref is called by the amsart/book/proc definition of \MR.
\providecommand{\MRhref}[2]{%
  \href{http://www.ams.org/mathscinet-getitem?mr=#1}{#2}
}
\providecommand{\href}[2]{#2}

\end{document}